\documentclass[11pt,a4paper,english,subeqn]{amsart}

\usepackage[T1]{fontenc}
\usepackage[utf8]{inputenc}
\usepackage{amsmath}
\usepackage{amssymb}
\usepackage{mathrsfs}
\usepackage{amsthm}
\usepackage{latexsym}
\usepackage{amsfonts}

\usepackage[onehalfspacing]{setspace} 
\setdisplayskipstretch{.421} 
\newlength{\abovebis} 
\newlength{\belowbis} 
\newlength{\aboveshortbis} 
\newlength{\belowshortbis} 
\everydisplay\expandafter{% 
  \the\everydisplay 
  \advance\abovedisplayskip\abovebis 
  \advance\belowdisplayskip\belowbis 
  \advance\abovedisplayshortskip\aboveshortbis 
  \advance\belowdisplayshortskip\belowshortbis 
}

\allowdisplaybreaks[1]

\def\R{\mathbb{R}}

\def\N{\mathbb{N}}
\def\C{\mathbb{C}}
\def\matn{M_{n}(\C)}

\def\Ree{\mathrm{Re}}
\def\Imm{\mathrm{Im}}
\def\ME{\mathscr{M}_E}
\def\Sc{\mathcal{S}}

\theoremstyle{plain}

\newtheorem{lem}{Lemma}[section]
\newtheorem{theo}[lem]{Theorem}
\newtheorem{prop}[lem]{Proposition}

\theoremstyle{definition}

\newtheorem{prob}{Problem}
\newtheorem{algo}{Algorithm}
\newtheorem*{algos}{Algorithms 1 and 2}
\newtheorem{rem}[lem]{Remark}

\numberwithin{equation}{section}

\begin{document}
\title[Monochromatic reconstruction algorithms]{Monochromatic reconstruction algorithms for two-dimensional multi-channel inverse problems}
\author{Roman G. Novikov}
\author{Matteo Santacesaria}
\address[R. G. Novikov and M. Santacesaria]{CNRS (UMR 7641), Centre de Math\'ematiques Appliqu\'ees, \'Ecole Polytechnique, 91128, Palaiseau, France}
\email{novikov@cmap.polytechnique.fr, santacesaria@cmap.polytechnique.fr}
\begin{abstract}
We consider two inverse problems for the multi-channel two-dimensional Schr\"odinger equation at fixed positive energy, i.e. the equation $-\Delta \psi + V(x)\psi = E \psi$ at fixed positive $E$, where $V$ is a matrix-valued potential. The first is the Gel'fand inverse problem on a bounded domain $D$ at fixed energy and the second is the inverse fixed-energy scattering problem on the whole plane $\R^2$. We present in this paper two algorithms which give efficient approximate solutions to these problems: in particular, in both cases we show that the potential $V$ is reconstructed with Lipschitz stability by these algorithms up to $O(E^{-(m-2)/2})$ in the uniform norm as $E \to +\infty$, under the assumptions that $V$ is $m$-times differentiable in $L^1$, for $m \geq 3$, and has sufficient boundary decay.
\end{abstract}

\maketitle
\section{Introduction}
We consider the equation
\begin{equation} \label{equa}
-\Delta \psi + V(x)\psi = E \psi, \qquad x \in \R^2, \qquad E >0,
\end{equation}
where
\begin{gather} \label{assv}
V \text{ is a sufficiently regular $\matn$-valued function on } \R^2 \\ \nonumber
\text{ with sufficient decay at infinity},
\end{gather}
$\matn$ is the set of the $n \times n$ complex matrices. This equation will also be considered on a domain $D$, where 
\begin{equation} \label{condd}
D \text{ is an open bounded domain in } \R^2 \text{ with a }C^2 \text{ boundary}.
\end{equation}

Equation \eqref{equa} at fixed $E$ can be considered as rather general multi-channel Schr\"odinger (resp. acoustic) equation on $D$ at a fixed energy (resp. frequency) related to $E$. It arises, in particular, as a 2D approximation to the following 3D equation
\begin{equation} \label{3dequa}
-\Delta_{x,z} \psi + v(x,z)\psi = E \psi, \quad (x,z) \in \Omega = D \times L,
\end{equation}
where $L = [a,b],\; a,b \in \R$, $v$ is a sufficiently regular complex-valued function on $\Omega$ and $\psi|_{D \times \partial L} = 0$ (for example): see \cite[Sec. 2]{NS2}. In this framework, the approximate 2D matrix-valued potential $V$ is given by
\begin{equation} \label{back}
V_{ij}(x) = \lambda_i \delta_{ij} + \int_L \bar \phi_i(z) v(x,z) \phi_j(z) dz, \qquad x \in D,
\end{equation}
for $1 \leq i,j \leq n$, where $n \in \N$, $\{ \phi_j \}_{j \in \N}$ is the orthonormal basis of $L^2(L)$ given by the eigenfunctions of $- \frac{d^2}{dz^2}$ such that $\phi_j|_{\partial L} = 0$, $-\frac{d^2 \phi_j}{d z^2}= \lambda_j \phi_j$, for $j \in N$, and $\delta_{ij} = 1$ if $i = j$ and $0$ otherwise.

In addition, equation \eqref{equa} can be seen as a particular case of the 2D Schr\"odinger equation in an external Yang-Mills field.\smallskip

For equation \eqref{equa} on $D$ we consider the Dirichlet-to-Neumann map $\Phi(E)$ such that
\begin{equation} \label{defphi}
\Phi(E)(\psi|_{\partial D}) = \left.\frac{\partial \psi}{\partial \nu}\right|_{\partial D}
\end{equation}
for all sufficiently regular solution $\psi$ of \eqref{equa} on $\bar D = D \cup \partial D$, where $\nu$ is the outer normal of $\partial D$.
Here we assume also that
\begin{equation} \label{direig}
E \textrm{ is not a Dirichlet eigenvalue for the operator } - \Delta + V \textrm{ in } D.
\end{equation}

This construction gives rise to the following inverse boundary value problem on $D$:

\begin{prob} \label{pro1}
Given $\Phi(E)$, find $V$ on $D$.
\end{prob}

On the other hand, for equation \eqref{equa} on $\R^2$, under assumptions \eqref{assv}, we consider the scattering amplitude $f$ defined as follows: we consider the continuous solutions $\psi^+(x,k)$ of \eqref{equa}, where $k$ is a parameter, $k \in \R^2, k^2 =E$, such that
\begin{align} \label{asymp}
\psi^+(x,k) &= e^{ikx}I-i\pi\sqrt{2 \pi} e^{-i\frac{\pi}{4}}f\left(k, |k|\frac{x}{|x|}\right) \frac{e^{i|k||x|}}{\sqrt{|k||x|}}\\ \nonumber
&\qquad + o \left(\frac{1}{\sqrt{|x|}} \right), \qquad \text{as } |x| \to \infty,
\end{align}
for some \textit{a priori} unknown $\matn$-valued function $f$, where $I$ is the identity matrix. The function $f$ on $\mathscr{M}_E = \{ (k,l) \in \R^2 \times \R^2 : k^2 = l^2 = E \} $ arising in \eqref{asymp} is the scattering amplitude for the potential $V$ in the framework of equation \eqref{equa}.

This construction gives rise to the following inverse scattering problem on $\R^2$:

\begin{prob} \label{pro2}
Given $f$ on $\mathscr{M}_E$, find $V$ on $\R^2$.
\end{prob}

Problems \ref{pro1} and \ref{pro2} can be considered as multi-channel fixed-energy analogues in dimension $d=2$ of inverse problems formulated in \cite{G} in dimension $d \geq 2$. Note that Problems 1 and 2 are not overdetermined, in the sense that we consider the reconstruction of a $\matn$-valued function $V$ of two variables from $\matn$-valued inverse problem data dependent on two variables. In addition, the history of inverse problems for the two-dimensional Schr\"odinger equation at fixed energy goes back to \cite{DKN} (see also \cite{N2, Gr} and reference therein). Note also that Problem \ref{pro1} can be considered as a model problem for the monochromatic ocean tomography (e.g. see \cite{BBS} for similar problems arising in this tomography).

As regards efficient algorithms for solving Problems \ref{pro1} and \ref{pro2} for the scalar case, i.e. for $n=1$, see \cite{N1, N2, N3,N4}. In addition, as concerns numerical implementations of these algorithms for Problem \ref{pro2} for $n=1$, see \cite{BBMRS, BAR}, and references therein.

Nevertheless, the fixed-energy global uniqueness for Problem \ref{pro1} (and for Problem \ref{pro2} with compactly supported $V$) for $n=1$ was completely proved only recently in \cite{B}. The reconstruction scheme of \cite{B} is not optimal with respect to its stability properties, and, therefore, is not efficient numerically in comparison with the aforementioned 2D reconstructions of \cite{N1,N2,N3,N4}, but it is very efficient for proving some global mathematical results. In particular: a related global logarithmic stability estimate for Problem \ref{pro1} for $n=1$ was proved in \cite{NS1}; global uniqueness and reconstruction results for Problem \ref{pro1} for $n \geq 2$ were obtained in \cite{NS2}; a global logarithmic stability estimate for Problem \ref{pro1} for $n \geq 2$ was proved in \cite{S}. In addition, Problem \ref{pro2} with compactly supported $V$ can be reduced, for $n \geq 2$, to Problem \ref{pro1}, as in \cite{N1} for $n=1$. This implies, at least, global uniqueness for Problem \ref{pro2} (in the compactly supported case). On the other hand, the uniqueness for Problem \ref{pro2} fails already for scalar ($n=1$) real-valued spherically-symmetric potentials $V$ of the Schwartz class on $\R^2$ (see \cite{GN}).\smallskip

The main purpose of the present work consists in generalizing the aforementioned reconstruction approach of \cite{N3, N4} to the case of Problems \ref{pro1} and \ref{pro2} for $n \geq 2$. As well as for $n=1$ this functional analytic approach gives an efficient non-linear approximation $V_{appr}(x,E)$ to the unknown $V(x)$ of Problems \ref{pro1} and \ref{pro2}. The reconstruction of $V_{appr}(x,E)$ from $\Phi(E)$ for Problem \ref{pro1} and from $f$ on $\ME$ for Problem \ref{pro2} is realized with some Lipschitz stability and is based on solving linear integral equations; see Algorithms 1 and 2 of Section 3, Theorems \ref{theo1}, \ref{theo2} and Remarks \ref{64}, \ref{65} of Section \ref{sec6}. Among these linear integral equations, the most important ones arise from a non-local Riemann-Hilbert problem. For the scalar case, Riemann-Hilbert problems of such a type go
back to \cite{M}. Another important part of these equations is used for transforming $\Phi(E)$ for Problem \ref{pro1} and $f$ on $\ME$ for Problem \ref{pro2} into $\matn$-valued Faddeev function analogues $h_{\pm}$ on $\ME$, involved in the formulation of the above-mentioned Riemann-Hilbert problem. In addition,
\begin{equation} \nonumber
\| V_{appr}(\cdot , E) - V\| = \varepsilon(E)
\end{equation}
rapidly decays as $E \to +\infty$, where $\| \cdot \|$ denotes an appropriate norm. In par\-ticular, $\varepsilon(E) = O(E^{- \infty})$ as $E \to + \infty$ if $\| \cdot \|$ is specified as $\| \cdot \|_{L^{\infty}(D)}$ and $V \in C^{\infty}(\R^2, \matn)$, $\mathrm{supp}\, V \subset D$, for Problem \ref{pro1} and if $\| \cdot \|$ is specified as $\| \cdot \|_{L^{\infty}(\R^2)}$ and $V \in \Sc(\R^2, \matn)$, for Problem \ref{pro2}, where $\Sc$ denotes the Sch\-wartz class.
In addition, no reconstruction algorithms for Problems 1 and 2 --- comparable, with respect to their stability, with Algorithms 1 and 2 and with an approximation error decaying more rapidly than $O(E^{-\frac 1 2})$ as $E \to \infty$ --- are available in the preceding literature, even for $V \in C^{\infty}(\R^2, \matn)$, $\mathrm{supp} V \subset D \subset \R^2$, when $n \geq 2$ (in general).

In spite of the fact that some excellent properties of Algorithms 1 and 2 are proved assuming that $V$ is sufficiently smooth and that $E$ is sufficiently great in comparison with (some norm of) $V$, we expect that these algorithms will work rather well even for $V$ with discontinuities and for the case when $E$ is not very big in comparison with $V$. This expectation is based on numerical results for Algorithm 2 for the case $n=1$; see \cite{BAR} and references therein. Numerical implementations of Algorithm 1 for $n \geq 1$ and Algorithm 2 for $n \geq 2$ are in preparation.\smallskip

Let us emphasize that in the present work we also develop studies of \cite{NS2} on the 2D multi-channel approach to 3D monochromatic inverse problems for equation \eqref{3dequa}. In this connection, the principal advantage of the 2D multi-channel Algorithm 1 (see section \ref{secalg})  in comparison with the 3D algorithm of \cite{N6} is that Algorithm 1 deals with non-overdetermined data and is only based on linear integral equations. High energy error estimates for both cases are similar. However, properties of Algorithm 1 of the present work are not estimated yet with respect to the approximation level $n$ in the framework of 3D applications.

Finally, note that multi-channel inverse problems and their applications to inverse problems in greater dimensions were initially considered for the one-dimensional multi-channel case, see \cite{AM}, \cite{ZS}. As one of the most recent result in this direction see \cite{K}.\smallskip

{\bf Acknowledgements.} The first author was partially supported by  the Russian Federation Government
grant No. 2010-220-01-077. We thank V. A. Burov, O. D. Rumyantseva, S. N. Sergeev and A. S. Shurup for very useful discussions.

\section{Faddeev functions}

In this section we recall some preliminary definitions.

Under assumptions \eqref{assv}, we consider the Faddeev functions $G(x,k) = e^{ikx}g(x,k)$, $\psi(x,k)$, $h(k,l)$ and related function $R(x,y,k)$ (see \cite{F, F2, N1, N5} for $n=1$):
\begin{align} \label{defg}
&g(x,k)=-\left(\frac{1}{2 \pi}\right)^2 \int_{\R^2} \frac{e^{i\xi x}}{\xi^2 + 2k\xi}d \xi,\\ \label{defpsi}
&\psi(x,k)=e^{ikx}I + \int_{\R^2} G(x-y,k)V(y)\psi(y,k)dy, \\ \label{defh}
&h(k,l)= \left(\frac{1}{2 \pi}\right)^2 \int_{\R^2}e^{-ilx}V(x)\psi(x,k) dx,\\ \label{defR}
&R(x,y,k) = G(x-y,k)+\int_{\R^2} G(x-\xi,k)V(\xi)R(\xi,y,k) d\xi
\end{align}
where $x=(x_1,x_2), y=(y_1,y_2) \in \R^2$, $k=(k_1,k_2) \in \C^2 \setminus \R^2$, $l =(l_1,l_2) \in \C^2$, $\Imm k = \Imm l \neq 0$ and $I$ is the identity matrix.
We recall that
\begin{align} \label{greendelta}
(\Delta + k^2)G(x,k)&=\delta(x),
\end{align}
for $x \in \R^2$, $k \in \C^2 \setminus \R^2$, where $\delta$ is the Dirac delta.
In addition: formula \eqref{defpsi} at fixed $k$ is considered as an equation for
\begin{equation} \label{defmu}
\psi(x,k) = e^{ikx}\mu (x,k),
\end{equation}
where $\mu$ is sought in $L^{\infty}(\R^2, \matn)$; formula \eqref{defR} at fixed $k$ and $y$ is considered as an equation for
\begin{equation}
R(x,y,k) = e^{ik(x-y)}r(x,y,k),
\end{equation}
where $r$ is sought in $L^2_{\mathrm{loc}}(\R^2, \matn)$, with the property that $|r(x,y,k)| \to 0$ as $|x| \to \infty$.
As a corollary of \eqref{defg}, \eqref{defpsi} and \eqref{greendelta}, $\psi$ satisfies \eqref{equa} for $E = k^2 = k_1^2+k_2^2$ and 
\begin{equation}
(\Delta + k^2-V(x))R(x,y,k) = \delta(x-y),
\end{equation}
for $x,y, \in \R^2$, $k \in \C \setminus \R^2$.
In addition, $h$ in \eqref{defh} is a generalised \textit{scattering amplitude} in the complex domain for the potential $V$.

For $\gamma \in S^1 = \{ \gamma \in \R^2 : |\gamma| = 1 \}$, we consider
\begin{gather}
G_{\gamma}(x,k) = G(x,k + i0 \gamma), \\ \label{defRg}
R_{\gamma}(x,y,k) = R(x,y,k + i0 \gamma), \\ \label{defpsig}
\psi_{\gamma}(x,k) =e^{ikx}\mu_{\gamma}(x,k), \qquad \mu_{\gamma}(x,k) = \mu(x,k + i 0 \gamma), \\
h_{\gamma}(k,l) = h(k + i 0 \gamma, l + i0 \gamma),
\end{gather}
where $x,y \in \R^2$, $k \in \R^2$, $l \in \R^2$. 

In addition, the functions
\begin{gather} \label{defGp}
G^+ (x,k) = G_{k / |k|}(x,k) = - \frac i 4 H^1_0 (|x||k|),\\ \label{defRp}
R^+(x,y,k) = R_{k / |k|}(x,y,k) \\ \label{defpsip}
\psi^+(x,k) = e^{ikx} \mu^+(x,k), \qquad \mu^+(x,k)= \mu_{k / |k|} (x,k), \\  \label{feq}
f(k,l)= h_{k / |k|}(k,l),
\end{gather}
for $x,y,k,l \in \R^2$, $|k| = |l|$, are functions from the classical scattering theory; in particular, $f$ is the scattering amplitude of \eqref{asymp} and $H^1_0$ is the Hankel function of the first type.
We also define
\begin{gather}\label{defhpm}
h_{\pm}(k,l) = h_{\pm \hat k_{\bot}}(k,l), \\ \label{defpsipm}
\mu_{\pm}(x,k)= \mu_{\pm \hat k_{\bot}}(x,k), \qquad \psi_{\pm}(x,k)=\psi_{\pm \hat k_{\bot}}(x,k), \\ \label{defRpm}
R_{\pm} (x,y,k) = R_{\pm \hat k_{\bot}}(x,y,k),
\end{gather}
where $k,l,x,y \in \R^2$, $|k| =|l|$, $\hat k_{\bot}= |k|^{-1}(-k_2,k_1)$ for $k= (k_1,k_2)$. Note that $\mu_+ \neq \mu^+$, $\psi_+ \neq \psi^+$ and $R_+ \neq R^+$ in general. We shall consider, in particular, the following restriction of the function $h$:
\begin{equation} \label{defb}
b(k)= h(k,-\bar k), \qquad \text{for } k \in \C^2, \; k^2= E > 0.
\end{equation}

We now introduce the notations
\begin{align} \nonumber
&z = x_1 + i x_2,& \bar z = x_1 - i x_2,& \\ \label{notations}
&\frac{\partial}{\partial z} = \frac 1 2 \left( \frac{\partial}{\partial x_1} - i\frac{\partial}{\partial x_2}  \right),& \frac{\partial}{\partial \bar z} = \frac 1 2 \left( \frac{\partial}{\partial x_1} + i\frac{\partial}{\partial x_2}  \right),& \\ \nonumber
&\lambda = E^{-1/2}(k_1 + ik_2),& \lambda' = E^{-1/2}(l_1 +i l_2),&
\end{align}
where $x=(x_1,x_2) \in \R^2$, $k=(k_1,k_2),l=(l_1,l_2) \in \C^2$, $k^2 = l^2 = E \in \R_+$. In the new notations
\begin{subequations}
\begin{gather} \label{defk}
k_1 = \frac 1 2 E^{1/2} (\lambda + \lambda^{-1}), \qquad k_2 = \frac i 2 E^{1/2}(\lambda^{-1}-\lambda), \\
l_1 = \frac 1 2 E^{1/2} (\lambda' + {\lambda'}^{-1}), \qquad l_2 = \frac i 2 E^{1/2}({\lambda'}^{-1}-\lambda'), \\
\exp(ikx) = \exp[\frac i 2 E^{1/2}(\lambda \bar z+\lambda^{-1}z )],
\end{gather}
\end{subequations}
where $\lambda, \lambda' \in \C \setminus \{0\}$, $z \in \C$ and the Schr\"odinger equation \eqref{equa} takes the form
\begin{equation}
-4\frac{\partial^2}{\partial z \partial \bar z}\psi+V(z)\psi = E \psi, \qquad z \in \C.
\end{equation}

In addition, the functions $f$ from \eqref{asymp} and \eqref{feq}, $h_{\pm}$ from \eqref{defhpm}, $\mu^+, \psi^+$ from \eqref{defpsip}, $\mu_{\pm}, \psi_{\pm}$ from \eqref{defpsipm}, $\psi$ from \eqref{defpsi}, $\mu$ from \eqref{defmu} and $b$ from \eqref{defb} take the form
\begin{align} \nonumber
f &=f(\lambda, \lambda', E),& h_{\pm} &=h_{\pm}(\lambda, \lambda', E), \\
\mu^+ &= \mu^+(z,\lambda,E),& \psi^+ &= \psi^+(z,\lambda,E), \\ \nonumber
\mu_{\pm} &= \mu_{\pm}(z,\lambda,E),& \psi_{\pm} &= \psi_{\pm}(z,\lambda,E),
\end{align}
where $\lambda, \lambda' \in T, \; z \in \C, \; E \in \R_+$,
\begin{equation}
\mu= \mu(z, \lambda, E), \qquad \psi= \psi(z, \lambda, E), \qquad b = b(\lambda,E),
\end{equation}
where $\lambda \in \C \setminus T, \; z \in \C, E \in \R_+$. Here
\begin{equation} \label{defT}
T = \{ \zeta \; : \; \zeta \in \C, |\zeta | = 1 \}.
\end{equation}

Under assumption \eqref{assv}, for $E$ sufficiently large the function $\mu(z,\lambda, E)$ has the following properties (see \cite{N3, N4} for $n=1$ and Section 4 for $n \geq 2$):
\begin{align} \label{propmu0}
&\mu(z,\lambda,E) \text{ is continuous in } \lambda \in \C \setminus T; \\ \label{propmu1}
&\mu(z,\lambda(1\mp 0), E) = \mu_{\pm}(z,\lambda,E) \text{ for } \lambda \in T; \\ \label{mupm}
&\mu_{\pm}(z,\lambda,E) = \mu^+(z,\lambda,E) \\ \nonumber
&\qquad + \pi i \int_T \mu^+(z,\lambda'',E)  \chi_+ \left(\pm i \left(\frac{\lambda}{\lambda''} - \frac{\lambda''}{\lambda} \right) \right) h_{\pm}(\lambda, \lambda'',z,E)   |d\lambda''|,
\end{align}
for $\lambda \in T$, where 
\begin{gather} \label{defchi}
\chi_+(s) = 0 \text{ for } s < 0, \quad  \chi_+(s) = 1 \text{ for } s \geq 0, \\ \label{defhhh}
h_{\pm}(\lambda,\lambda',z,E) = \exp\left[ -\frac i 2 E^{1/2} \left( \lambda \bar z + \frac{z}{\lambda} - \lambda' \bar z - \frac{z}{\lambda'} \right) \right] h_{\pm}(\lambda, \lambda',E); \\ \label{eqdbar}
\frac{\partial}{\partial \bar \lambda}\mu(z,\lambda,E) = \mu \left(z, - \frac{1}{\bar \lambda},E\right) r(\lambda,z,E),
\end{gather}
for $\lambda \in \C \setminus T$, where
\begin{equation} \label{defr}
r(\lambda, z, E)= \exp \left[ - \frac i 2 E^{1/2}\left( \lambda \bar z + \frac{z}{\lambda} + \bar \lambda z + \frac{\bar z}{\bar \lambda} \right) \right] \frac{\pi}{\bar \lambda} \mathrm{sign}(\lambda \bar \lambda -1) b(\lambda,E),
\end{equation}
where $b$ is defined by means of \eqref{defb} and \eqref{defk};
\begin{gather} \label{devmu}
\mu(z,\lambda,E) = I + \frac{\mu_{-1}(z,E)}{\lambda} + o\left( \frac{1}{\lambda} \right), \qquad \lambda \to \infty,\\ \label{formv2}
V(z)= 2iE^{1/2}\frac{\partial}{\partial z}\mu_{-1}(z,E).
\end{gather}
The following formula is valid (see \cite{N4} for $n=1$ and Section 4 for $n \geq 2$):
\begin{align} \label{formv}
V(z)&=2iE^{1/2} \frac{\partial}{\partial z}\left( \frac{1}{\pi} \int\limits_{D_-}\mu(z,-\frac{1}{\bar \zeta},E) r(\zeta,z,E)d \Ree \zeta \, d \Imm \zeta\right. \\ \nonumber
&\qquad +\frac{1}{2 \pi i} \left. \int\limits_T \mu_-(z,\zeta,E) i \zeta |d \zeta| \right), 
\end{align}
for $z \in \C$, $E$ sufficiently large and $D_- = \{ \zeta \; : \; \zeta \in \C, |\zeta | > 1 \}$.

\section{Reconstruction algorithms} \label{secalg}
We present here Algorithms 1 and 2, which yield approximate but sufficiently stable solutions to Problems 1 and 2, respectively. These algorithms have a final common part: the reconstruction of the approximate potential $V_{appr}$ starting from $h_{\pm}$ of \eqref{defhpm}. Thus, for the sake of clarity, we first give the different initial parts of the algorithms---that is, the reconstruction of $h_{\pm}$ starting from $\Phi(E)$ for Algorithm 1 and from $f$ for Algorithm 2---and then the final common part.

Note that in both algorithms we consider in particular the functions $\psi_{\pm}, h_{\pm}, \mu_-$ of \eqref{defhpm}, \eqref{defpsipm} and $\mu^+$ of \eqref{defpsip}. In addition, in Algorithm 1, in the definitions of these functions we assume that $V \equiv 0$ on $\R^2 \setminus D$.

\begin{algo}[$\Phi(E) \longrightarrow h_{\pm}$]
Given $\Phi(E)$, for $E$ sufficiently large, we first reconstruct $\psi_{\pm}(x,k)|_{\partial D},k \in \R^2, k^2 =E$, with the help of the following Fredholm linear integral equation (see \cite{N1} for $n=1$ and Section 4 for $n\geq 2$):
\begin{equation}\label{eq11}
\psi_{\pm}(x,k)|_{\partial D}=e^{ikx}I+\int_{\partial D}A_{\pm}(x,y,k)\psi_{\pm}(y,k) dy, \quad k \in \R^2, \; k^2 = E,
\end{equation}
where
\begin{align} \label{eq12}
&A_{\pm}(x,y,k)=\int_{\partial D} G_{\pm}(x-\xi,k) \left(\Phi - \Phi_0 \right)(\xi,y,E) d\xi, \quad x,y \in \partial D, \\ \label{eq13}
&G_{\pm}(x,k) =G^+(x,k) - \frac{1}{4 \pi i} \int_{S^1}e^{i|k|\theta x}\chi_+(\pm \theta k_{\bot}) d\theta,
\end{align}
$I$ is the identity matrix, $(\Phi-\Phi_0)(x,y,E)$ is the Schwartz kernel of the operator $\Phi(E)-\Phi_0(E)$, $\Phi_0(E)$ is the Dirichlet-to-Neumann operator associated to the zero potential in $D$ at fixed energy $E$, $G^+(x,k)$ is defined in \eqref{defGp}, $k_{\bot}= (-k_2, k_1)$ for $k = (k_1,k_2)$, $dy, \; d \xi$ denote the standard Euclidean measure on the boundary $\partial D$ and $d\theta$ denotes the standard Euclidean measure on $S^1$.

Then, in order to obtain $h_{\pm}$, it is sufficient to use the following formula (see \cite{N1} for $n=1$ and Section 4 for $n\geq 2$):
\begin{equation} \label{rech}
h_{\pm}(k,l) = \frac{1}{(2 \pi)^2} \int_{\partial D} \int_{\partial D}e^{-ilx}(\Phi-\Phi_0)(x,y,E)\psi_{\pm}(y,k)dy dx, \quad (k,l) \in \ME.
\end{equation}
\end{algo}

\begin{algo}[$f \longrightarrow h_{\pm}$]
Starting from $f$ on $\ME$ (for $E$ sufficiently large), one directly recovers $h_{\pm}$ solving the following integral equation (see \cite{N2, N4} for $n=1$ and Section 4 for $n \geq 2$):
\begin{align} \label{hpmf}
&h_{\pm}(\lambda,\lambda',E)-\pi i \int\limits_T   f(\lambda'',\lambda',E) \chi_+\left(\pm i \left(\frac{\lambda}{\lambda''}-\frac{\lambda''}{\lambda} \right) \right) h_{\pm}(\lambda, \lambda'',E)   |d \lambda''| \\ \nonumber
&\qquad = f(\lambda,\lambda',E), \qquad (\lambda,\lambda') \in T \times T.
\end{align}
\end{algo}

\begin{algos}[$h_{\pm} \longrightarrow V_{appr}$]
We begin with the construction of $\tilde \mu^+$, an approximation to $\mu^+$ of \eqref{defpsip}; this is done by solving the following integral equation arising from the non-local Riemann-Hilbert problem \eqref{propmu0}-\eqref{devmu} for $\mu$ in the approximation that $b \equiv 0$ at fixed $E$ (see \cite{N4} for $n=1$ and Section 4 for $n \geq 2$):
\begin{align} \label{mupapp}
\tilde \mu^+(z,\lambda,E) +\int_T \tilde \mu^+(z,\lambda',E) B(\lambda,\lambda',z,E)  |d\lambda'| =I, \; \lambda \in T, \; z \in \C,
\end{align}
where $E$ is sufficiently large and
\begin{align} \label{defB}
&B(\lambda,\lambda', z,E)= \frac 1 2 \int_T h_-(\zeta,\lambda',z,E) \chi_+\left(-i \left(\frac{\zeta}{\lambda'}-\frac{\lambda'}{\zeta}\right) \right) \frac{d\zeta}{\zeta-\lambda(1-0)} \\ \nonumber
&\qquad - \frac 1 2 \int_T h_+(\zeta,\lambda',z,E) \chi_+\left(i \left(\frac{\zeta}{\lambda'} -\frac{\lambda'}{\zeta}\right) \right) \frac{d\zeta}{\zeta-\lambda(1+0)},
\end{align}
where $\chi_+$, $h_{\pm}$ are defined in \eqref{defchi}, \eqref{defhhh}.
Then one can obtain an approximation $\tilde \mu_{-}$ to $\mu_-$ via \eqref{mupm}, used as follows:
\begin{align} \label{recmum}
&\tilde \mu_{-}(z,\lambda,E) = \tilde \mu^+(z,\lambda,E) \\ \nonumber
&\qquad + \pi i \int_T \tilde \mu^+(z,\lambda'',E)  \chi_+ \left(- i \left(\frac{\lambda}{\lambda''} - \frac{\lambda''}{\lambda} \right) \right) h_{-}(\lambda, \lambda'',z,E)  |d\lambda''|,
\end{align}
for $\lambda \in T$, $z \in \C$. Finally, the approximate potential $V_{appr}(\cdot, E)$ can be obtained using the following formula (see \cite{N3,N4} for $n=1$ and Section 4 for $n \geq 2$):
\begin{align} \label{fvapp}
V_{appr}(z,E)= 2 i E^{1/2} \frac{\partial}{\partial z} \left( \frac{1}{2 \pi i} \int_T \tilde \mu_-(z,\zeta,E) i\zeta |d \zeta | \right).
\end{align}
\end{algos}

The approximate potential $V_{appr}$ depends in a non-linear way on $\Phi(E)$ in Algorithm 1 and on $f$ on $\ME$ in Algorithm 2, in spite of the fact that both algorithms are based on solving linear integral equation. In the linear approximation near zero potential, the following formulas hold:
\begin{align}
&h_{\pm}(k,l) \approx \frac{1}{(2 \pi)^2}\int_{\partial D} \int_{\partial D}e^{i(-lx+ky)}(\Phi - \Phi_0)(x,y,E)dx\,dy, \quad (k,l) \in \ME,
\end{align}
for linearised Algorithm 1;
\begin{align} \label{apprh}
&h_{\pm}(\lambda, \lambda',E) \approx f(\lambda, \lambda',E), \qquad \lambda, \lambda' \in T,
\end{align}
for linearised Algorithm 2;
\begin{align}
&V_{appr}(z,E) \approx \frac 1 \pi E^{1/2} \int_T w(z,\lambda, E) i \lambda |d \lambda|,
\end{align}
where $z \in D$ for linearised Algorithm 1 and $z \in \C$ for linearised Algorithm 2, and
\begin{align} \label{apprw}
&w(z, \lambda, E) = \frac{\partial}{\partial z}\left( \pi i \int_T \exp\left[ -\frac i 2 E^{1/2} \left( \lambda \bar z + \frac{z}{\lambda} - \lambda' \bar z - \frac{z}{\lambda'} \right) \right] \right. \\ \nonumber
&\qquad \left. \times \, \mathrm{sign}\left(- i \left( \frac{\lambda}{\lambda'} - \frac{\lambda'}{\lambda} \right) \right) h_{\pm}(\lambda,\lambda',E) |d\lambda'| \right),
\end{align} 
for $z \in \C$, $\lambda \in T$, $E > 0$.
\subsection{Algorithm 1 with a non-zero background potential $\Lambda$} \label{nonzero}
Consider a potential $V$ defined as in \eqref{back}, where the diagonal matrix $\Lambda$, defined as $\Lambda_{ij} =  \lambda_i \delta_{ij}$, is supposed to be a known background potential. In this case Algorithm 1 admits the following effectivisations.

Let $V_1 \equiv \Lambda$ on $\bar D$, $V_1 \equiv 0$ on $\R^2 \setminus \bar D$. The following parts A and B provide two different approaches to the reconstruction of $\psi_{\pm}(x,k)|_{\partial D}$ from $\Phi(E)$ and of $h_{\pm}(k,l)$ from $\psi_{\pm}(x,k)|_{\partial D}$; the reconstruction of $V_{appr}$ from $h_{\pm}$ is given after in steps C and D.

{\bf A.} $\Phi(E) \longrightarrow h_{\pm}$. Starting from $\Phi(E)$, for $E$ sufficiently large, we first reconstruct $\psi_{\pm}(x,k)|_{\partial D}, k \in \R^2, k^2 =E$, with the help of the following Fredholm linear integral equation (see Section 4):
\begin{equation} \label{eqa112}
(\mathrm{Id} + (\mathrm{Id}-A^1_{\pm})^{-1}\delta A_{\pm})\psi_{\pm}(x,k)|_{\partial D} = \psi^1_{\pm}(x,k)|_{\partial D},
\end{equation}
where 
\begin{align}
&A^1_{\pm}u(x) = \int_{\partial D} A^1_{\pm}(x,y,k)u(y)dy, \quad x \in \partial D,\\
&A^1_{\pm}(x,y,k)=\int_{\partial D} G_{\pm}(x-\xi,k) \left(\Phi_1 - \Phi_0 \right)(\xi,y,E) d\xi, \quad x,y \in \partial D,\\
&\delta A_{\pm}u(x) = \int_{\partial D \times \partial D}\! \! \! \! \! \! \! \! \! \! \! \! G_{\pm}(x-\xi,k)(\Phi_1 - \Phi)(\xi,y,E)u(y)dy \, d\xi, \; x \in \partial D,
\end{align}
$\psi^1_{\pm}(x,k)|_{\partial D}= (\mathrm{Id}-A^1_{\pm})^{-1} (e^{ikx}I)$ are the functions $\psi_{\pm}(x,k)|_{\partial D}$ for $V=V_1$, $(\Phi_1-\Phi_0)(x,y,E)$ is the Schwartz kernel of the operator $\Phi_1(E)-\Phi_0(E)$, $(\Phi_1-\Phi)(x,y,E)$ is the Schwartz kernel of the operator $\Phi_1(E)-\Phi(E)$, $\Phi_0(E)$ is the Dirichlet-to-Neumann operator associated to the zero potential in $D$ at fixed energy $E$, $\Phi_1(E)$ is the Dirichlet-to-Neumann operator associated to the potential $V_1$ in $D$ at fixed energy $E$ and $u$ is a $\matn$-valued test function on $\partial D$.

In order to obtain $h_{\pm}$ we use the following formula (see Section 4):
\begin{align} \label{recah2}
&h_{\pm}(k,l) = h^1_{\pm}(k,l) + \frac{1}{(2 \pi)^2} \int_{\partial D} \int_{\partial D}e^{-ilx}(\Phi-\Phi_1)(x,y,E)\psi_{\pm}(y,k)dy dx \\ \nonumber
&\qquad + \frac{1}{(2 \pi)^2} \int_{\partial D} \int_{\partial D}e^{-ilx}(\Phi_1-\Phi_0)(x,y,E)\delta \psi_{\pm}(y,k)dy dx,
\end{align}
for $(k,l) \in \ME$, where $h^1_{\pm}(k,l)$ is defined as in \eqref{defh}, \eqref{defhpm} with $V=V_1$, $\delta \psi_{\pm}(x,k)=\psi_{\pm}(x,k)-\psi^1_{\pm}(x,k)$ and $\psi^1_{\pm}(x,k)$ is defined as $\psi_{\pm}(x,k)$ in \eqref{defpsi}, \eqref{defpsig}, \eqref{defpsipm} with $V= V_1$.

{\bf B.} $\Phi(E) \longrightarrow h_{\pm}$. 
As above, starting from $\Phi(E)$, for $E$ sufficiently large, we first reconstruct $\psi_{\pm}(x,k)|_{\partial D}, k \in \R^2, k^2 =E$, with the help of the following Fredholm linear integral equation (see \cite{N5} for $n=1$ and Section 4 for $n\geq 2$):
\begin{equation}\label{eqa11}
\psi_{\pm}(x,k)|_{\partial D}=\psi^1_{\pm}(x,k)|_{\partial D} +\int_{\partial D} A_{\pm}(x,y,k)\psi_{\pm}(y,k) dy,
\end{equation}
for $k \in \R^2, \; k^2 = E$, where
\begin{align} \label{eqa12}
&A_{\pm}(x,y,k)=\int_{\partial D} R^1_{\pm}(x,\xi,k) \left(\Phi - \Phi_1 \right)(\xi,y,E) d\xi, \quad x,y \in \partial D,
\end{align}
$\psi^1_{\pm}$, $R^1_{\pm}$ are defined as $\psi_{\pm}$, $R_{\pm}$ of \eqref{defpsi}, \eqref{defR}, \eqref{defRg}, \eqref{defpsig}, \eqref{defpsipm}, \eqref{defRpm} with $V = V_1$, $(\Phi-\Phi_1)(x,y,E)$ is the Schwartz kernel of the operator $\Phi(E)-\Phi_1(E)$, $\Phi_1(E)$ is the Dirichlet-to-Neumann operator associated to the potential $V_1$ in $D$ at fixed energy $E$.

In order to obtain $h_{\pm}$ we use the following formula (see \cite{N5} for $n=1$ and Section 4 for $n\geq 2$):
\begin{equation} \label{recah}
h_{\pm}(k,l) = h^1_{\pm}(k,l) + \frac{1}{(2 \pi)^2} \int_{\partial D} \int_{\partial D}\psi^1_{\mp}(x,-k,-l)(\Phi-\Phi_1)(x,y,E)\psi_{\pm}(y,k)dy dx,
\end{equation}
for $(k,l) \in \ME$, where $h^1_{\pm}(k,l)$ is defined as in \eqref{defh}, \eqref{defhpm} with $V=V_1$, $\psi^1_{\mp}(x,k,l)$ is defined as the solution of the following linear integral equation (see \cite{N5} for $n=1$ and Section 4 for $n\geq 2$)
\begin{equation} \label{pssii}
\psi^1_{\mp}(x,k,l) = e^{ilx}I + \int_{\R^2}G_{\mp} (x-y,k) V_1(y) \, \psi^1_{\mp}(y,k,l)dy,
\end{equation}
where $x,k,l \in R^2$, $k^2 = l^2 > 0$ and $G_{\mp}$ is defined in \eqref{eq13}.\smallskip

{\bf C.} $h_{\pm} \longrightarrow \tilde \mu^+$. We construct an approximation $\tilde \mu^+$ to $\mu^+$ of \eqref{defpsip} via the following integral equation which generalises \eqref{mupapp} (see Section 4):
\begin{align} \label{mupapp2}
(\mathrm{Id} +(\mathrm{Id}+B^1)^{-1}\delta B) \tilde \mu^+(z,\lambda,E) = \mu^{1,+}(z,\lambda,E),  \quad \lambda \in T, \; z \in \C,
\end{align}
for $E$ sufficiently large, where
\begin{align}
&B^1 u(\lambda) = \int_T u(\lambda') B^1(\lambda,\lambda',z,E)  |d\lambda'|,\\
&\delta B u(\lambda) = \int_T u(\lambda') [B(\lambda, \lambda', z,E)-B^1(\lambda,\lambda',z,E)]  |d\lambda'|,
\end{align}
%\begin{align} \label{mupapp2}
%&\tilde \mu^+(z,\lambda,E) +\int_T \tilde \mu^+(z,\lambda',E) B(\lambda,\lambda',z,E)  |d\lambda'| \\ \nonumber
%&\qquad = \mu^{0,+}(z,\lambda,E) +\int_T \mu^{0,+}(z,\lambda',E) B^0(\lambda,\lambda',z,E)  |d\lambda'| , \; \lambda \in T, \; z \in \C,
%\end{align}
for $\lambda \in T, \; z \in \C$, $B^1(\lambda,\lambda',z,E)$ is defined as $B(\lambda,\lambda',z,E)$ in \eqref{defB} with $h_{\pm} = h^1_{\pm}$, $\mu^{1,+}$ is defined as $\mu^{+}$ in \eqref{defmu}, \eqref{defpsip} with $V=V_1$ and $u$ is a $\matn$-valued test function on $T$.

{\bf D.} $\tilde \mu^+ \longrightarrow V_{appr}$. The final part of the algorithm is the same as for Algorithm 1 with zero background potential. We construct an approximation $\tilde \mu_-$ to $\mu_-$ using formula \eqref{recmum} and then the approximate potential $V_{appr}(z,E)$ via formula \eqref{fvapp}.\smallskip

In the linear approximation near the potential $V_1$, the following formulas hold:
\begin{subequations}
\begin{align}
&h_{\pm}(k,l) \approx h^1_{\pm}(k,l) \\ \nonumber 
&\quad + \frac{1}{(2 \pi)^2} \int_{\partial D} \int_{\partial D}e^{-ilx}(\Phi-\Phi_{1})(x,y,E)\psi^1_{\pm}(y,k)dy dx \\ \nonumber
&\quad -\frac{1}{(2 \pi)^2} \int_{\partial D} \int_{\partial D}e^{-ilx}(\Phi_1-\Phi_{0})(x,y,E)(\mathrm{Id}-A^1_{\pm})^{-1}\delta A_{\pm}\psi^1_{\pm}(y,k)dy\, dx,\\
&h_{\pm}(k,l) \approx h^1_{\pm}(k,l) \\ \nonumber 
&\quad + \frac{1}{(2 \pi)^2} \int_{\partial D} \int_{\partial D}\psi^1_{\mp}(x,-k,-l)(\Phi-\Phi_{1})(x,y,E)\psi^1_{\pm}(y,k)dy dx,
\end{align}
\end{subequations}
for $(k,l) \in \ME$,
\begin{align}
\tilde \mu^+(z,\lambda,E) \approx \mu^{1,+}(z,\lambda,E) - (\mathrm{Id}+B^1)^{-1}\delta B \mu^{1,+}(z,\lambda,E),
\end{align}
for $\lambda \in T, \; z \in \C$,
\begin{align}
&\tilde \mu_{-} (z,\lambda, E) \approx \mu^1_{-}- (\mathrm{Id}+B^1)^{-1}\delta B \mu^{1,+}(z,\lambda,E) \\ \nonumber
&\quad +  \pi i \int_T \mu^{1,+}(z,\lambda'',E)  \chi_+ \left(- i \left(\frac{\lambda}{\lambda''} - \frac{\lambda''}{\lambda} \right) \right)(h_{-}-h^1_{-})(\lambda, \lambda'',z,E)  |d\lambda''|, \\
&V_{appr}(z,E) \approx V_1 - \frac 1 \pi E^{1/2} \int_T \frac{\partial}{\partial z}\left( (\mathrm{Id}+B^1)^{-1}\delta B \mu^{1,+}(z,\lambda,E) \right) i \lambda |d \lambda| \\ \nonumber
&\qquad + i E^{1/2} \int_T  \int_T  \frac{\partial}{\partial z}\Big[ \mu^{1,+}(z,\lambda'',E)  \chi_+ \left(- i \left(\frac{\lambda}{\lambda''} - \frac{\lambda''}{\lambda} \right) \right)  \\ \nonumber 
&\qquad \qquad \times (h_{-}-h^1_{-})(\lambda, \lambda'',z,E) \Big] |d\lambda''| i \lambda |d \lambda|,
\end{align}
for $z \in D$ and $E$ sufficiently large.
%\begin{align}
%&V_{appr}(z,E) \approx \frac 1 \pi E^{1/2} \int_T w(z,\lambda, E) i \lambda |d \lambda|, \quad z \in D,\\
%&w(z,\lambda,E) = \frac{\partial}{\partial z}\left\{ \pi i \int_T \left[- \nu^0(z,\lambda',E) \chi_+ \left(+ i \left( \frac{\lambda}{\lambda'} - \frac{\lambda'}{\lambda} \right) \right) h_{+}(\lambda, \lambda',z, E) \right. \right. \\ \nonumber
%&\qquad \left. \left. + \nu^0(z,\lambda',E) \chi_+ \left(- i \left( \frac{\lambda}{\lambda'} - \frac{\lambda'}{\lambda} \right) \right) h_{-}(\lambda, \lambda',z, E)\right] |d \lambda'|  + \nu^0(z,\lambda,E) \right\}, \\ \nonumber
%&\nu^0(z,\lambda,E) = \mu^{0,+}(z,\lambda,E) +\int_T \mu^{0,+}(z,\lambda',E) B^0(\lambda,\lambda',z,E)  |d\lambda'|,
%\end{align}
%for $z \in \C$, $\lambda \in T$, $E > 0$.

\section{Derivation of some formulas and equations of \\ Section 2 and 3 for the matrix case}
The following formula and equations will be useful:
\begin{align} \label{eq1}
&\psi_{\gamma}(x,k) = \psi^+ (x,k) \\ \nonumber
&\qquad + 2 \pi i \int_{\R^2} \psi^+(x,l)\delta(l^2-k^2) \chi_+((l-k)\gamma)h_{\gamma}(k,l) dl,\\ \label{eq2}
&h_{\gamma}(k,l) = f(k,l) \\ \nonumber
&\qquad + 2 \pi i \int_{\R^2}f(m,l) \delta(m^2-k^2)\chi_+((m-k)\gamma)h_{\gamma}(k,m)dm,
\end{align}
for $\gamma \in S^1$, $x,k,l \in \R^2$, $k^2=E \in \R_+$ sufficiently large,
\begin{align} \label{eq3}
& \frac{\partial \mu}{\partial \bar k_j}(x,k)= - 2 \pi \int_{\R^2}\xi_j e^{i\xi x}\mu(x,k+\xi) H(k,-\xi) \delta(\xi^2+2k\xi) d\xi,\\ \label{eq4}
&\frac{\partial H}{\partial \bar k_j} (k,p) = - 2 \pi \int_{\R^2}\xi_j H(k+\xi,p+\xi)H(k,-\xi) \delta(\xi^2 + 2k\xi) d\xi,
\end{align}
for $j =1,2$, $k \in \C^2 \setminus \R^2$, $k^2=E \in \R_+$ sufficiently large, $x,p \in \R^2$, where $H(k,p) = h(k,k-p)$, $\delta$ is the Dirac delta and the other functions were already defined in Section 2. Formula \eqref{eq1} and equations \eqref{eq2}-\eqref{eq4} are proved in \cite{F2, BC, HN} for the scalar case: the proof can be straightforwardly generalized to the matrix case, where one only has to pay attention to the order of factors (which is indeed different from the formulation given in the quoted papers, but coherent with similar results obtained in \cite{X}).\smallskip

Now formula \eqref{mupm} follows directly from \eqref{eq1}, \eqref{defmu}, \eqref{defhhh} using notations \eqref{notations}; equation \eqref{eqdbar} follows from \eqref{eq3} taking into account \eqref{defb}, \eqref{defr} and notations \eqref{notations}. In addition, equation \eqref{hpmf} is a direct consequence of \eqref{eq2} (with $\gamma = \pm \hat k_{\bot}$) using notations \eqref{notations}.

Formula \eqref{formv} follows from \eqref{formv2}, \eqref{devmu}, \eqref{eqdbar}, \eqref{propmu1}, \eqref{propmu0} (these can be proved exactly as in the scalar case) and the Cauchy--Pompeiu formula
\begin{equation}
u(\lambda) = \frac{1}{2 \pi i} \int_{\partial \mathcal{D}} u(\zeta) \frac{d \zeta}{\zeta - \lambda} -\frac 1 \pi \int_{\mathcal{D}} \frac{\partial u(\zeta)}{\partial \bar \zeta} \frac{d \Ree \zeta \, d \Imm \zeta}{\zeta - \lambda}, \quad \lambda \in \mathcal{D},
\end{equation}
for any sufficiently regular $\matn$-valued $u$ in $\mathcal{D},$ where $\partial \mathcal{D}$ is sufficiently regular. In addition, formula \eqref{fvapp} is just formula \eqref{formv} without the first term in the sum.

Equation \eqref{mupapp} is an approximation of the following (exact) equation for $\mu^+$:
\begin{equation} \label{mup}
\mu^+(z,\lambda,E) +\int_T \mu^+(z,\lambda',E) B(\lambda,\lambda',z,E)  |d\lambda'| =I + \varphi(z,\lambda,E),
\end{equation}
for $\lambda \in T, \; z \in \C$, where
\begin{equation} \label{deffi}
\varphi(z,\lambda,E) = - \frac 1 \pi \int_{\C}\mu\left(z, -\frac{1}{\bar \zeta}, E\right) r(\zeta,z,E) \frac{d \Ree \zeta \, d \Imm \zeta}{\zeta - \lambda}.
\end{equation}
The derivation of \eqref{mup} can be found in \cite{N4} for the scalar case and its generalisation to the matrix case is straightforward (paying attention to the order of factors).\smallskip

Formula \eqref{eq13} is a result of \cite{F2}, while formulas \eqref{eq11}, \eqref{eq12}, \eqref{rech} are results of \cite{N1} for the scalar case and can be proved for the matrix case following the scheme of \cite{NS2}, where similar formulas appear.\smallskip

Formulas \eqref{eqa112} and \eqref{recah2} follows from \eqref{eq11} and \eqref{rech}.

Formulas \eqref{eqa11}-\eqref{pssii} are results of \cite{N5} for the scalar case and can directly extended to the matrix case following the scheme of \cite{NS2} because, in particular, the general matrix version of Alessandrini's identity in \cite{NS2} works for our diagonal background potential $\Lambda$.

Finally, equation \eqref{mupapp2} follows from \eqref{mup}. 
%Finally, equation \eqref{mupapp2} follows from \eqref{mup} and
%\begin{equation} \label{mup0}
%\mu^{0,+}(z,\lambda,E) +\int_T \mu^{0,+}(z,\lambda',E) B^0(\lambda,\lambda',z,E)  |d\lambda'| =I + \varphi^0(z,\lambda,E),
%\end{equation}
%for $\lambda \in T, \; z \in \C$, where $\varphi^0(z,\lambda,E)$ is defined as $\varphi(z,\lambda,E)$ in \eqref{deffi} in the case of the potential $\Lambda$. We can indeed write $\varphi = \varphi^0 + \delta \varphi$, then substitute equation \eqref{mup0} into \eqref{mup} and forget the term $\delta \varphi$.

\section{Function spaces and some estimates}
We introduce some function spaces, which will be useful to prove the high-energy convergence of our algorithms. For $m \in \N, \; \varepsilon >0$ we consider
\begin{gather*}
W^{m,1}(\R^2, M_{n}(\C)) = \{ u : \partial^k u \in L^1(\R^2, \matn) \text{ for } |k| \leq m \}, \\
W^{m,1}_{\varepsilon}(\R^2, M_{n}(\C)) = \{ u : \varkappa^{\varepsilon}\partial^k u \in L^1(\R^2, \matn) \text{ for } |k| \leq m \}, \\
(\varkappa^{\varepsilon}u)(x) = (1+|x|^2)^{\varepsilon /2} u(x), \qquad k \in (\N \cup 0)^2, \qquad |k| = k_1 + k_2, \\
\partial^k = \partial_1^{k_1} \partial_2^{k_2}, \qquad \partial_j = \frac{\partial}{\partial x_j};
\end{gather*}
for $ \alpha \in ] 0, 1]$, $s \in \R$ we consider
\begin{gather*}
C^{\alpha,s}(\R^2, \matn) = \{u \; : \; \| u\|_{\alpha,s} < \infty \},
\end{gather*}
where
\begin{gather*}
\|u\|_{\alpha, s} = \| \varkappa^s u\|_{\alpha}\\ 
\|w\|_{\alpha}= \sup_{p, \xi \in \R^2, \; |\xi| \leq 1} \left(|w(p)| +\frac{|w(p+\xi)-w(p)|}{|\xi|^{\alpha}}\right),\\
(\varkappa^s u)(p) = (1+ |p|^2)^{s/2} u(p), \qquad |u(p)| = \max_{1 \leq i,j \leq n} |u_{ij}(p)|;
\end{gather*}
in addition we consider $\mathcal{H}_{\alpha,s}(\R^2, \matn)$, defined as the closure of $C^{\infty}_0 (\R^2, \matn)$ (the space of infinitely smooth functions with compact support) in $\| \cdot \|_{\alpha, s}$.

Let
\begin{equation}
\widehat V (p) = \frac{1}{(2 \pi)^2}\int_{\R^2}e^{ipx}V(x) dx, \qquad p \in \R^2.
\end{equation}
If a matrix-valued potential $V$ satisfies
\begin{align} \label{condv2}
\quad V \in W^{m,1}_{\varepsilon}(\R^2, \matn)\; \text{for some } \varepsilon >0, \; m \in \N,
\end{align}
then
\begin{align} \label{condv}
\widehat V \in \mathcal{H}_{\alpha,s}(\R^2, \matn), \qquad \alpha \in ]0,1], \qquad s \in \R_+,
\end{align}
where $\alpha = \min(1,\varepsilon)$, $s = m$. Let 
\begin{align}
\Sigma(r) = (1-r)^{-1} r.
\end{align}
We have the following results:
\begin{prop} \label{propest}
Let the condition \eqref{condv} be valid. Then
\begin{subequations}
\begin{align}
&|f(k,l)-\widehat V(k-l)| \leq \Sigma(r) \| \widehat V \|_{\alpha,s} (1+ |k-l|^2)^{-s/2}, \\
&|H_{\gamma}(k,p)-\widehat V(p)| \leq \Sigma(r) \| \widehat V\|_{\alpha,s} ( 1+ p^2)^{-s/2},
\end{align}
for $r = |k|^{-\sigma} c_1(\alpha,s,\sigma, n)\|\widehat V \|_{\alpha,s} < 1,\; k,l,p \in \R^2, \; \gamma \in S^1,\; k^2 \geq 1$,
\begin{align}
|H(k,p)-\widehat V(p)| \leq \Sigma(r)\| \widehat V \|_{\alpha, s} (1+p^2)^{-s/2},
\end{align}
\end{subequations}
for $r= |\Ree k|^{-\sigma}  c_1(\alpha,s,\sigma, n)\|\widehat V \|_{\alpha,s} < 1,\; k \in \C^2 \setminus \R^2,\; p \in \R^2, \; \R \ni k^2 \geq 1$.
In particular
\begin{subequations} \label{estff}
\begin{align}
|f(k,l)| &\leq 2 \| \widehat V \|_{\alpha,s} (1+|k-l|^2)^{-s/2}, \qquad k,l \in \R^2, \\
|H_{\gamma}(k,p)| &\leq 2 \| \widehat V \|_{\alpha,s} (1+p^2)^{-s/2}, \qquad k,p \in \R^2, \qquad \gamma \in S^1, \\
|H(k,p)| &\leq 2 \| \widehat V \|_{\alpha,s} (1+p^2)^{-s/2}, \qquad k \in \C^2 \setminus \R^2, \qquad p \in \R^2,
\end{align}
\end{subequations}
for $k^2 \geq E_1 = \max(1, (2 c_1 (\alpha,s,\sigma,n) \| \widehat V \|_{\alpha,s})^{2/ \sigma})$, where $H_{\gamma}(k,l) =\newline h_{\gamma}(k, k -l)$, $0 < \alpha < 1$, $s > 0$, $0 < \sigma < \min(1, s)$.
\end{prop}
%This result is proved in \cite{N4} for the scalar case and its generalisation to the matrix case is straightforward. The next lemma, which will be useful in the next section, follows immediately from Proposition \ref{propest} and the following formulas (see \cite{HN} for the scalar case; the matrix case follows exactly the same proof):
%\begin{align}
%\mu_{\gamma}(x,k) = 1 - \int_{\R^2}\frac{e^{i\xi x}H_{\gamma}(k,-\xi)}{\xi^2+ 2 ( k + i0\gamma) \xi}d \xi,
%\end{align}
%for $x,k \in \R^2$, $\gamma \in S^1$;
%\begin{align}
%\mu(x,k) = 1 - \int_{\R^2}\frac{e^{i\xi x}H(k,-\xi)}{\xi^2+ 2 k \xi}d \xi,
%\end{align}
%for $x \in \R^2$, $k \in \C^2 \setminus \R^2$.
%
\begin{lem} \label{lem52}
Under condition \eqref{condv}, we have the following estimates:
\begin{subequations} \label{estmuu}
\begin{align}
|\mu_{\gamma}(x,k)-I| + \left| \frac{\partial \mu_{\gamma}(x,k)}{\partial x_1}\right|+ \left| \frac{\partial \mu_{\gamma}(x,k)}{\partial x_2}\right| \leq |k|^{-\sigma}c_2(\alpha,s,\sigma,n)\| \widehat V \|_{\alpha,s},
\end{align}
for $x=(x_1, x_2) \in \R^2$, $k \in \R^2$, $\gamma \in S^1$,
\begin{align}
|\mu(x,k)-I| + \left| \frac{\partial \mu(x,k)}{\partial x_1}\right|+ \left| \frac{\partial \mu(x,k)}{\partial x_2}\right| \leq |\Ree k|^{-\sigma}c_2(\alpha,s,\sigma,n)\| \widehat V \|_{\alpha,s},
\end{align}
for $x=(x_1, x_2) \in \R^2$, $k \in \C^2 \setminus \R^2$, $k^2 \geq E_1(\alpha,s,\sigma,n,\| \widehat V\|_{\alpha,s})$, where $0<\alpha<1$, $s > 1$, $0 < \sigma < \min(1, s-1)$.
\end{subequations}
\end{lem}

Proposition \ref{propest} and Lemma \ref{lem52} for the scalar case ($n=1$) were given in \cite{N4} and their generalisation to the matrix case ($n \geq 2$) is straightforward.

\section{Lipschitz stability and rapid convergence of \\ Algorithms 1 and 2 for $E \to +\infty$} \label{sec6}
We present here main rigorous results concerning stability and convergence of our algorithms in the case of zero background potential for simplicity. In addition, we expect that, for potentials of the form \eqref{back}, Algorithm 1 with non-zero background potential $\Lambda$ (see subsection \ref{nonzero}) will work even better than its version with zero background potential.
\begin{theo}[Stability and convergence of Algorithm 1] \label{theo1}
Let $V \in W^{m,1}(\R^2, \matn)$, $m \geq 3$, $\mathrm{supp} \, V \subset D$ and let $\Phi(E)$ be the Dirichlet-to-Neumann operator of \eqref{defphi} at fixed energy $E$, where $E \geq E_2(\alpha,s,\sigma,n,\| \widehat V\|_{\alpha,s})$, $0 < \alpha \leq 1$, $s=m$, $0 < \sigma < 1$ and $E$ is not a Dirichlet eigenvalue of $-\Delta + V$ and $-\Delta$ in $D$. Then $V$ is reconstructed from $\Phi(E)$ with Lipschitz stability via Algorithm 1 up to $O(E^{-(m-2)/2})$ in the uniform norm as $E \to + \infty$.
\end{theo}

\begin{theo}[Stability and convergence of Algorithm 2] \label{theo2}
Let $V$ satisfy \eqref{condv2}, for $m \geq 3$, and let $f$ be the scattering amplitude of \eqref{asymp} at fixed energy $E \geq E_2(\alpha,s,\sigma,n,\| \widehat V\|_{\alpha,s})$, where $\alpha = \min(1,\varepsilon)$, $s=m$ and $0 < \sigma < 1$. Then $V$ is reconstructed from $f$ on $\ME$ with Lipschitz stability via Algorithms 2 up to $O(E^{-(m-2)/2})$ in the uniform norm as $E \to + \infty$.
\end{theo}

The constant $E_2$ of Theorems \ref{theo1} and \ref{theo2} is precisely stated in Remark \ref{63}.
The Lipschitz stability of Theorems \ref{theo1} and \ref{theo2} is specified in the proofs of these theorems and is summarized in Remarks \ref{64} and \ref{65}. The error term $O(E^{-(m-2)/2})$ of Theorems \ref{theo1} and \ref{theo2} is made explicit in formula \eqref{apprv}.\smallskip

Similarly with the presentation of Algorithms 1 and 2 in section 3, we separate the proofs of Theorems \ref{theo1} and \ref{theo2} in several steps.
\begin{proof}[Proof of Theorem \ref{theo1} ($\Phi(E) \longrightarrow h_{\pm}$)]
We have that equation \eqref{eq11} is a Fredholm linear integral equation of second kind for $\psi_{\pm}|_{\partial D} \in L^2(\partial D)$, which is uniquely solvable with precise data $\Phi - \Phi_0$ (the proof of the latter fact is the same as in the scalar case; see \cite{N1}). Therefore the reconstruction of $\psi_{\pm}$ via \eqref{eq11} is Lipschitz stable, with respect to small errors in $\Phi - \Phi_0$ (in the $L^2$ norm of the Schwartz kernel), 

As a corollary, the reconstruction of $h_{\pm}$ in $L^2(\ME)$ from $\Phi(E)-\Phi_0(E)$ via equation \eqref{eq11} and formula \eqref{rech} is also Lipschitz stable.
\end{proof}

\begin{proof}[Proof of Theorem \ref{theo2} ($f \longrightarrow h_{\pm}$)]
Estimates \eqref{estff} and notations \eqref{notations} give
\begin{subequations}
\begin{align}
|f(\lambda, \lambda',E)| &\leq 2 \| \widehat{V}\|_{\alpha,s} ( 1 + E |\lambda-\lambda'|^2)^{-s/2}, \qquad \lambda, \lambda' \in T,\\ \label{estf}
\|f\|_{L^2(T \times T)}&\leq c_3 n \| \widehat{V}\|_{\alpha,s} E^{-1/4},
\end{align}
\end{subequations}
for $E \geq E_1$, $\alpha = \min(1,\varepsilon)$, $s=m$. Now, under the assumptions of Theorem \ref{theo2}, integral equation \eqref{hpmf} is uniquely solvable for $h_{\pm}(\lambda, \cdot, E) \in L^2(T)$ for $\lambda \in T$, $E \geq E_1$ (this is a consequence of the unique solvability of integral equation \eqref{defpsi} for $E \geq E_1$). In addition, by estimate \eqref{estf}, for $E \geq \max(E_1, (\pi c_3 n \| \widehat{V}\|_{\alpha,s})^4)$, equation \eqref{hpmf} is uniquely solvable for $h_{\pm}(\lambda, \cdot, E) \in L^2(T)$, $\lambda \in T$, and for $h_{\pm}(\cdot , \cdot, E) \in L^2(T \times T)$ by the method of successive approximations. This implies the Lipschitz stability of the reconstruction of $h_{\pm}$ on $T \times T$ from $f$ on $T \times T$, with respect to small errors in the $L^2$ norm.
\end{proof}

\begin{proof}[Proof of Theorems \ref{theo1} and \ref{theo2} ($h_{\pm} \longrightarrow V_{appr}$)]
The proof follows as in the scalar case (that was treated in \cite{N4}), except for the order of the terms in formulas and integral equations.\smallskip
%Using the same notations as in the algorithms, we have that integral equation \eqref{mupapp} for $\tilde{\mu}^+$ is stable with respect to small errors in $h_{\pm}$ (with respect to its $L^2$ norm).

Estimates \eqref{estff}, formula \eqref{feq} and notations \eqref{notations} give
\begin{subequations} \label{esthhh}
\begin{align}
|h_{\pm}(\lambda, \lambda', E)| &\leq 2 \| \widehat{V}\|_{\alpha,s} ( 1 + E |\lambda-\lambda'|^2)^{-s/2},\qquad \lambda, \lambda' \in T,\\ \label{esth}
\|h_{\pm}\|_{L^2(T \times T)}&\leq c_3 n \| \widehat{V}\|_{\alpha,s} E^{-1/4},
\end{align}
\end{subequations}
for $E \geq E_1$, $s=m$ and
\begin{equation} \label{alp}
0 < \alpha \leq 1 \text{ for Theorem \ref{theo1} and } \alpha = \min(1,\varepsilon) \text{ for Theorem \ref{theo2}}.
\end{equation}
We define the integral operator $B(z,E)$ as
\begin{equation}
(B(z,E)u)(\lambda) = \int_T  u(\lambda') B(\lambda, \lambda',z,E) |d\lambda'|,
\end{equation}
for $\lambda \in T$, where $B(\lambda,\lambda',z,E)$ is defined in \eqref{defB} and $u$ is a test matrix function. The following decomposition holds
\begin{equation} \label{dec}
B(z,E) = C_+ Q_-(z,E) - C_- Q_+ (z,E),
\end{equation}
where
\begin{align}
(C_{\pm}u)(\lambda) &= \frac{1}{2 \pi i} \int_T \frac{u(\zeta)}{\zeta - \lambda( 1 \mp 0) } d \zeta, \\
(Q_{\pm} u )(\lambda) &= \pi i \int_T u (\lambda') \chi_+ \left( \pm i \left( \frac{\lambda}{\lambda'} - \frac{\lambda'}{\lambda} \right) \right) h_{\pm}(\lambda, \lambda',z, E) |d \lambda'|,
\end{align}
$z \in \C$, $\lambda \in T$, $\chi_+, h_{\pm}$ are defined in \eqref{defchi} and \eqref{defhhh} and $u$ is a test matrix function. Thanks to \eqref{esthhh}, \eqref{dec} and properties of the Cauchy projectors $C_{\pm}$ (see \cite{N4} for more details), $B(z,E)$ satisfies the estimates
\begin{subequations}  \label{estB}
\begin{align} \label{estBa}
\| B(z,E)u\|_{L^2(T)} &\leq c_4 n \| \widehat{V}\|_{\alpha,s} E^{-1/4} \|u\|_{L^2(T)}, \\
\left\| \frac{\partial}{\partial z}B(z,E)u \right\|_{L^2(T)} &\leq c_4 n  \| \widehat{V}\|_{\alpha,s} E^{-1/4} \|u\|_{L^2(T)},
\end{align}
\end{subequations}
for $z \in \C$, $E \geq E_1$, $s=m$, $\alpha$ as in \eqref{alp}.\smallskip 

Now by estimate \eqref{estBa}, for $E \geq \max(E_1, (c_4 n \| \widehat{V}\|_{\alpha,s})^4)$, integral equation \eqref{mupapp} is uniquely solvable for $\tilde{\mu}^+ (z, \cdot, E) \in L^2(T)$, at fixed $z \in \C$, by the method of successive approximations. This implies the Lipschitz stability of the reconstruction of $\tilde{\mu}^+(z, \cdot,E)$ on $T$, at fixed $z \in \C$, from $h_{\pm}$ on $T \times T$ with respect to small errors in the $L^2$ norm.
\end{proof}
\begin{proof}[Proof of Theorems \ref{theo1} and \ref{theo2} ($V_{appr} \longrightarrow V$)]
Our high-energy convergence estimate is as follows:
\begin{equation} \label{apprv}
|V(z) - V_{appr}(z,E)| \leq c_5 n \| \widehat{V}\|_{\alpha, s} E^{-(s-2)/2},
\end{equation}
where $z \in \C$, $E \geq E_2(\alpha,s,\sigma,n,\| \widehat{V}\|_{\alpha, s})$, $\alpha$ as in \eqref{alp}, $s=m$, $0< \sigma <1$ (see \cite{N4} for complete details). This estimate follows from \eqref{mupm}, \eqref{formv}, \eqref{mupapp}--\eqref{fvapp}, \eqref{esth} and the following estimates (whose proofs for $n=1$ can be found in \cite{N4}):
\begin{gather} 
\left| 2 i E^{1/2} \frac{\partial}{\partial z} \left( \int_{D_-}  \mu(z,-\frac{1}{\bar \zeta},E) r(\zeta, z,E) d\Ree \zeta\, d \Imm \zeta \right) \right| \leq c_6 n  \| \widehat{V}\|_{\alpha,s} E^{-(s-2)/2}, \\ \label{apprmu} 
\| \mu^+(z,\cdot, E) - \tilde \mu^+(z,\cdot, E)\|_{L^2(T, \matn)} \leq c_7 n \| \widehat{V}\|_{\alpha, s} E^{-s/2}, \\
\left\| \frac{\partial \mu^+}{\partial z}(z,\cdot, E) - \frac{\partial \tilde \mu^+}{\partial z}(z,\cdot, E)\right\|_{L^2(T, \matn)}\! \! \! \! \! \! \! \! \! \! \! \leq c_7 n \| \widehat{V}\|_{\alpha, s} E^{-(s-1)/2},
\end{gather}
for $z \in \C$, $s = m \geq 3$, $E \geq E_2(\alpha,s,\sigma,n,\| \widehat{V}\|_{\alpha, s})$,  $\alpha$ as in \eqref{alp}.
\end{proof}

\begin{rem} \label{63}
The constant $E_2$ of Theorems \ref{theo1} and \ref{theo2} can be fixed as some constant such that $E \geq E_2$ implies that
\begin{gather*}
E \geq E_1, \qquad |\mu(z,\lambda,E)| \leq 2, \qquad \left| \frac{\partial}{\partial z}\mu (z,\lambda,E)\right| \leq 1, \\
\|B(z,E)\|_{L^2(T)}^{op} \leq \frac{1}{2}, \qquad  \left\|\frac{\partial }{\partial z}B(z,E) \right\|_{L^2(T)}^{op} \leq \frac{1}{2},
\end{gather*}
for $z \in \C$, $\lambda \in \C$, where $\mu$ and $B$ are estimated in \eqref{estmuu} and \eqref{estB}.
\end{rem}

Now, let $\Phi_{V,0}(x,y,E), \; x,y \in \partial D$, denote the Schwartz kernel of the operator $\Phi(E)-\Phi_0(E)$ considered as precise data for Problem 1. Let $\Phi_{V,0}'$ denote $\Phi_{V,0}$ with some small errors (for the case of Problem 1) and $f'$ denote $f$ with some small errors (for the case of Problem 2). Let $V'_{appr}$ denote $V_{appr}$ reconstructed from $\Phi_{V,0}'$ via Algorithm 1 (for Problem 1) and from $f'$ via Algorithm 2 (for Problem 2).

The Lipschitz stability of Theorems \ref{theo1} and \ref{theo2} is summarized in the following remarks:

\begin{rem} \label{64}
Let the assumptions of Theorem \ref{theo1} hold and let
\begin{equation}
\delta = \|\Phi_{V,0}'(\cdot,\cdot,E) - \Phi_{V,0}(\cdot,\cdot,E) \|_{L^2(\partial D \times \partial D)} \leq \delta_1(V,E,D,n).
\end{equation}
Then
\begin{equation}
\varepsilon = \|V_{appr}' - V_{appr}\|_{L^{\infty}(D)} \leq \eta_1(V,E,D,n) \delta.
\end{equation}
Here $\delta_1$ and $\eta_1$ are some positive constants summarizing the Lipschitz stability of Algorithm 1. In particular,
\begin{align}
\delta_1(V,E,D,n) &\geq \delta_1^0, \\ \label{eta1}
\eta_1(V,E,D,n) &\leq \eta_1^0 E,
\end{align}
as $\| \Phi_{V,0}(\cdot, \cdot, E) \|_{L^2(\partial D \times \partial D)} \to 0$, for some positive (sufficiently small) $\delta_1^0$ and (sufficiently big) $\eta_1^0$, where $\delta_1^0$ and $\eta_1^0$ are independent of $V$ and $E$ for fixed $D$ and $n$.

\end{rem}
\begin{rem} \label{65}
Let the assumptions of Theorem \ref{theo2} hold and let
\begin{equation} \label{errf}
\delta = \|f - f'\|_{L^2(\ME)} \leq \delta_2(V,E,n).
\end{equation}
Then
\begin{equation}
\varepsilon = \|V_{appr} - V_{appr}'\|_{L^{\infty}(\R^2)} \leq \eta_2(V,E,n) \delta.
\end{equation}
Here $\delta_2$ and $\eta_2$ are suitable constants summarizing the Lipschitz stability of Algorithms 2. In particular,
\begin{align} \label{delta2}
\delta_2(V,E,n) &\geq \delta_2^0, \\ \label{eta2}
\eta_2(V,E,n) &\leq \eta_2^0 E,
\end{align}
as $\|f\|_{L^2(\ME)} \to 0$, for some positive (sufficiently small) $\delta_2^0$ and (sufficiently big) $\eta_2^0$, where $\delta_2^0$ and $\eta_2^0$ are independent of $V$ and $E$ for fixed $n$.
\end{rem}

Note that in Remark \ref{65}, the norm $\| \cdot \|_{L^2(\ME)}$ is identified with $\| \cdot \|_{L^2(T \times T)}$.

The property that $\|f\|_{L^2(\ME)} \to 0$, mentioned in Remark \ref{65}, is fulfilled, in particular, for $E \to + \infty$, as a consequence of estimate \eqref{estf}. On the contrary, the property that $\| \Phi_{V,0}(\cdot, \cdot, E)\|_{L^2(\partial D \times \partial D)} \to 0$, mentioned in Remark \ref{64}, is not fulfilled for $E=E_j, j \to \infty$, for any sequence $\{ E_j \}_{j \in \N}$ of positive real numbers such that $E_j \to +\infty$ as $j \to \infty$, if $V \not \equiv 0$. In this connection, our high-energy conjecture is that 
\begin{equation}
\sup_{j \in \N} \| \Phi_{V,0}(\cdot, \cdot, E_j)\|_{L^2(\partial D \times \partial D)} < +\infty,
\end{equation}
for some $\{ E_j \}_{j \in \N}$ dependent on $V$, where $E_j \to +\infty$ as $j \to \infty$.

Note that the $E$ factor in the right side of \eqref{eta1} and of \eqref{eta2} is related with the choice of the $L^2$ norm for estimates of the inverse problem data. 
For example, for Algorithm 2, at least in the linear approximation \eqref{apprh}-\eqref{apprw}, this factor disappear if $\| \cdot \|_{L^2(\ME)}$ is replaced by $\| \cdot \|_{L^{\infty}_s(\ME)}$, $s =m$, where
\begin{equation}
\|u \|_{L^{\infty}_s(\ME)} = \sup_{(\lambda, \lambda') \in T \times T} (1+E|\lambda - \lambda'|^2)^{s/2} |u(\lambda, \lambda')|.
\end{equation}

\end{document}